\documentclass[11pt,leqno]{article}

\usepackage{amssymb,amsmath,amsthm, url,rotating}
 \usepackage{graphicx,epsfig}
 \usepackage{xcolor,graphicx}
\newcommand{\n}{\noindent}

\newcommand{\vp}{\varepsilon}
\newcommand{\bb}[1]{\mathbb{#1}}
\newcommand{\cl}[1]{\mathcal{#1}}

\newcommand{\ovl}{\overline}

\theoremstyle{plain}
\newtheorem{thm}{Theorem} 
\newtheorem{lem}[thm]{Lemma}

\newtheorem{cor}[thm]{Corollary}

\theoremstyle{definition}

\newtheorem{dfn}[thm]{Definition}

\theoremstyle{remark}
\newtheorem{rem}[thm]{Remark}

\def\tilde{\widetilde}

\renewcommand{\tilde}{\widetilde}

\def\R{\bb R}
\def\Z{\bb Z}

\def\C{\bb C}

\def\E{\bb E}

\def\N{\bb N}
\def\P{\bb P}

\def\T{\bb T}

\def\nl{\nolimits}
\begin{document}
\def\d{\delta}

 \title{
  A note on Sidon sets in bounded orthonormal systems}

\author{by\\
 Gilles  Pisier\\
Texas A\&M University and UPMC-Paris VI}

\maketitle
\begin{abstract}  
We give a simple example of an $n$-tuple
of orthonormal elements in $L_2$
(actually martingale differences)
bounded by a fixed constant, and hence subgaussian with
a fixed constant but that are Sidon
only with constant $\approx \sqrt n$. This is
optimal. The first example of this kind was given by Bourgain and Lewko,
but with constant $\approx \sqrt {\log n}$.
We also include the analogous $n\times n$-matrix valued example,
for which the optimal constant is $\approx  n$.
We deduce from our example  that there are two
$n$-tuples each Sidon with constant 1,
lying in orthogonal linear subspaces and such that
their union is Sidon only with constant $\approx \sqrt n$.
This is again asymptotically optimal.
We show that  any martingale difference sequence
with values in $[-1,1]$ is ``dominated" in a natural sense
(related to our results) by any 
 sequence of independent, identically distributed, symmetric $\{-1,1\}$-valued variables (e.g. the Rademacher
functions). We include a self-contained proof that any sequence $(\varphi_n)$ 
that is the union
of two Sidon sequences lying in orthogonal subspaces is
 such that   $(\varphi_n\otimes\varphi_n \otimes\varphi_n\otimes\varphi_n)$ is Sidon.
 \end{abstract}
 
 MSC: 43A46,42C05 (primary).  60642, 60646 (secondary).

 One of the most celebrated results
 in the theory of Sidon sets in the trigonometric system
 on the circle
 (or   on a compact Abelian group) is Drury's union theorem
 that says that the union of two (disjoint) Sidon sets is still a Sidon set.
 In a recent paper Bourgain and Lewko \cite{BoLe} 
 considered Sidon sets for a general
  uniformly bounded  orthonormal system $(\varphi_n )$
  in $L_2$ over an arbitrary probability space $(T,m)$. They
 extended some of the classical  results known
 for systems of  characters on compact Abelian groups.
 We continued on the same theme in \cite{Pi3}.
 Let us recall the basic definitions.
 We say that $(\varphi_n )$ is Sidon if there is a constant $C $
  such that for any finitely supported scalar sequence $n\mapsto x_n$
  \begin{equation}\label{1bis}\sum |x_n| \le C \|\sum x_n \varphi_n\|_\infty.\end{equation}
  The smallest such $C$ is called the Sidon constant of $(\varphi_n )$.
  The system $(\varphi_n )$ is called
   $\otimes^k$-Sidon if the system ($\varphi_n (t_1 )\varphi_n (t_2 )\cdots \varphi_n (t_k )$)
   is Sidon in $L_2(T^k,m\times\cdots\times m)$.
   We say that
   $(\varphi_n )$ is subgaussian
   if there is a constant $\beta$ such that
   for any finite   scalar sequence $(x_n)$ such that $\sum |x_n|^2\le 1$
   we have
   $$\int  e^{|\sum x_n \varphi_n|^2/\beta^2} dm \le e.
   $$
   When this holds we say that $(\varphi_n )$ is $\beta$-subgaussian.\\
   Bourgain and Lewko \cite{BoLe} proved that subgaussian
   does not imply Sidon but does  imply
   $\otimes^5$-Sidon, and the author \cite{Pi3} improved this to
   $\otimes^2$-Sidon.  \\
      Let $(g_n)$ be an i.i.d. sequence of standard Gaussian random variables.
    We say that $(\varphi_n )$
    is randomly Sidon if there is a constant
    $C$ such that for any finite scalar sequence $(x_n)$ we have
  $$\sum |x_n|\le C \E\|\sum g_n x_n \varphi_n\|_\infty.$$
 In \cite{Pi3}, we
 proved that randomly Sidon implies
 $\otimes^4$-Sidon. It follows as an immediate corollary that
 the union of two  mutually orthogonal  Sidon systems  
 is $\otimes^4$-Sidon (see Theorem   \ref{n1} for a   quick outline of a
 direct proof). This generalizes Drury's celebrated
 union theorem for sets of characters. Naturally,
 this last result raises the question whether
 $\otimes^4$-Sidon can be replaced by $\otimes^k$-Sidon
 for $k<4$. While we cannot decide this for
 $k=2$ or $k=3$, the goal of the present note is to
 settle the question at least  for $k=1$.

We first   improve Bourgain and Lewko's example from 
   \cite{BoLe}   showing that subgaussian does not imply Sidon for
 uniformly bounded  orthonormal systems.
 Our example is a (very simple) martingale difference sequence
 and the constant is asymptotically sharp.
 As a corollary we show that, not surprisingly,
 Drury's union theorem
 does not extend to 
 two mutually orthogonal uniformly bounded  orthonormal systems.
 
 \def\tr{{\rm tr}}
 \begin{thm}\label{t1} Fix $\vp>0$. There is a uniformly bounded  
 real valued orthonormal system $(\varphi_n)$
 with
 $\|\varphi_n \|_\infty \le 1+\vp$ for all $n$
  that is  subgaussian 
  and actually  satisfies
\begin{equation}
\label{e2} \E e^{ \sum x_n \varphi_n} \le e^{(1+\vp)^2\sum x_n^2/2} 
\end{equation}
for any finite sequence of real numbers $(x_n)$,
but  $(\varphi_n)$ is not a Sidon system.\\
More precisely, 
 the smallest constant  $C_n$
such that for any scalar coefficients
$(x_k)$ 
we have
$$\sum\nl_1^n |x_k| \le C_n \|\sum\nl_1^n x_k \varphi_k\|_\infty$$
satisfies
\begin{equation}\label{e1}
\forall n\ge1\quad
C_n\ge \d_\vp\sqrt n,
\end{equation}
where  $\d_\vp>0$   depends only on $\vp$.
In addition, $(\varphi_n)$
is a  martingale difference sequence.
 \end{thm}
  \begin{proof} 
Let $(\vp_n)$ be a sequence of independent
choices of signs, i.e. independent $\pm$-valued
random variables on a probability space
$(\Omega,\P)$ taking the values $\pm 1$ with probabilility $1/2$.
Let $\cl A_n$ be the 
$\sigma$-algebra generated by 
$\{\vp_k\mid 0\le k\le n\}$.
Let $0=a_0  \le\cdots \le a_{n-1}\le a_n\le\cdots$ be a fixed 
non-decreasing sequence for the moment.
 Consider  $A_0=\Omega$,
 $S_0=0$,  
 and define inductively $A_n\in \cl A_n$
 and $S_n$ as follows:
 $$S_n=S_{n-1}+ \vp_n 1_{A_{n-1}} 
  \text{  and  }
   A_{n}=\{ |S_{n}|\le a_{n}  \}.$$
Assume that $\P(A_n)\ge \d$ for some fixed $\d>0$.
Then let
 \begin{equation}
 \label{e7} f_n= \vp_n 1_{A_{n-1}}.\end{equation}
This is a martingale difference sequence with $\|f_n\|_\infty\le 1$,
therefore an orthogonal system such that
$$S_n= f_1+\cdots+f_n$$
and moreover
$$\|f_n\|_2^2 \ge \d.$$
We claim that the Sidon constant of 
$\{f_1,\cdots,f_n\}$
is $\ge n/(1+a_{n-1})$.
This follows from the observation that
  \begin{equation}
 \label{e6}
\forall n\quad \|S_n\|_\infty \le 1+a_{n-1}.\end{equation}
Indeed, this is immediate 
by induction on $n$ (since either
 $\|S_n\|_\infty\le a_{n-1}+1$ or
$\|S_n\|_\infty \le \|S_{n-1}\|_\infty $
  depending whether
$\|S_n\|_\infty $ is attained on $A_{n-1}$ or on its complement).

 Now by Azuma's inequality (see e.g.  \cite[p. 501]{Pi4}) we know that
 $(f_n)$ is subgaussian with
 a good constant. In fact for any real numbers $t$ and 
 $x_n$ with
 $(x_n)$ in $\ell_2$
 \begin{equation}
 \label{e3}
 \E e^{t\sum x_n f_n} \le e^{t^2\sum |x_n|^2/2} .\end{equation}
 In particular
 \begin{equation}\label{e333}\P(\{ |S_n|> t\}) \le 2 e^{-t^2 /2n}.\end{equation}
 Fix $\vp>0$.
 Taking $a_n=c\sqrt n$, this gives us
  $$\P(\{ |S_n|> a_n\}) \le 2 e^{-c^2 /2},$$
  so we can choose a numerical value of $c$,
 namely $c=c_\vp$,  large enough so that 
   $$\P(\{ |S_n|> a_n\}) \le 1- (1+\vp)^{-2}.$$
   Then we have by what precedes
  $\|S_n\|_\infty\le a_{n-1}+1= c_\vp \sqrt {n-1}+1$ and
  $$\|f_n\|_2=\P(\{ |S_{n-1}|\le a_{n-1}  \})^{1/2} \ge (1+\vp)^{-1}$$
 for all $n$. Therefore  
   the Sidon constant of 
$\{f_1,\cdots,f_n\}$
is $\ge n/(1+a_{n-1})$.
Letting
$$\varphi_n=f_n \|f_n\|_2^{-1}$$
we find $\|\varphi_n \|_\infty \le 1+\vp$ for all $n$,
$ (\varphi_n) $ is orthonormal 
and \eqref{e1} holds.
By Azuma's inequality \eqref{e3}  we also have \eqref{e2}.
\end{proof} 
\begin{rem} I am grateful to B. Maurey
for   suggesting the following
neater example $(S'_k)$. 
Let us first fix $n\ge 1$, and hence $a_n>0$ is fixed.
Let $M_k=\vp_1+\cdots+\vp_k$ for all $k\ge 1$. Define the stopping time $T_n$
  by
$T_n=\inf\{k\ge 0\mid |M_k|>a_n\}$ and $T_n=\infty$ if $|M_k|\le a_n$ for all $k\ge 0$.
Recall the classical inequalities
$$\forall t>0\quad \P(\{  \sup_{1\le k\le n} |M_k|> t \} ) \le 2\P(\{   |M_n|> t \} )\le 4
e^{-t^2/2n}.$$
The first one goes back to  Paul L\'evy (see e.g. \cite[p. 28]{Pi4}),
it is closely related to D\'esir\'e Andr\'e's reflection principle for Brownian motion
(see e.g. \cite[p. 558]{Loe})
and the second one follows from \eqref{e333}.
We then set  for $k\ge 1$
$ S'_k=M_{k\wedge T_n}$ and $$f'_k=S'_k-S'_{k-1}=\vp_k 1_{\{ T_n\ge k \}}.$$
In the previous example this corresponds to sets $A'_{k-1}={\{ T_n\ge k \}}={\{ T_n\le k-1 \}^c}\in \cl A_{k-1}$.
We have clearly $\|S'_k\|_\infty \le a_n +1$ for all $k$,
and it is easy to check, 
since $A'_{k-1}=\{ \sup_{j<k} |M_j|\le a_n\}$,
that we again can choose $a_n= c_\vp \sqrt n$ so that
for any $1\le k\le n$ we have
$$\P(A'_{k-1})\ge \P(A'_{n})=\P(\{  \sup_{1\le k\le n} |M_k|\le a_n \} ) 
=1-\P(\{  \sup_{1\le k\le n} |M_k|> a_n \} ) \ge (1+\vp)^{-2}.$$

\end{rem}

\begin{rem} Since $(\varphi_n)$ is formed of mean zero variables
\eqref{e2} holds iff
  there is $\beta'$ such that
\begin{equation}
\forall p\ge 2 \ \forall  (x_n)\in \ell_2\quad
\|\sum x_n \varphi_n\|_p\le \beta'\sqrt p (\sum |x_n|^2)^{1/2}\end{equation}
\end{rem}
\begin{rem} 
Let  $(\varphi_n)$
be any orthonormal  system.
Then for any scalar coefficients $(x_k)$ we have obviously
$$\sum\nl_1^n |x_k|\le \sqrt n (\sum\nl_1^n |x_k|^2)^{1/2}
\le   \sqrt n \| \sum\nl_1^n x_k\varphi_k\|_\infty
.$$ Thus the order of growth of the Sidon constant
in \eqref{e1} and the next statement are both sharp.
\end{rem}
 \begin{cor} There are two 
 orthonormal martingale difference sequences
 $(\varphi^+_n)$
 and
 $(\varphi^-_n)$
with orthogonal linear spans
such that
each has the same distribution as the Rademacher functions
(i.e. each is formed of independent $\pm$-valued random variables
with mean zero) but their union is not a Sidon system.
More precisely the union of $\{\varphi^+_k\mid k\le n\}$
 and
 $\{\varphi^-_k\mid k\le n\}$
has a Sidon constant $C_n$ growing like $\sqrt n$.
\end{cor}
 \begin{proof}  
 Let  will modify slightly the preceding proof and construct by induction a sequence
 $S'_n$.
 We  wish to choose by induction a set $B_n\subset \Omega$
 in  $\cl A_n$ 
 (just like $A_n$ was)
 and we again set $S'_n=S'_{n-1}+ \vp_n 1_{B_{n-1}}$.
  but we choose $B_n$  satisfying
 \begin{equation}\label{e5} B_n \subset  \{|S'_n|\le a_n\}  \quad \text{ and } \P(B_n)=1/2.\end{equation}
 To be able to make this choice all
 we need to know  is that
 $\P(\{|S'_n|\le a_n\}) \ge1/2.$
 Then the preceding argument, associated to $\vp=\sqrt 2-1$
 still guarantees that $\P( \{|S'_{n-1}|\le a_{n-1}\}) \ge 1/2$.
 Thus we clearly can select  $B_n$
 for which \eqref{e5} holds and we again obtain
 $\|S'_n\|_\infty \le 1+\sqrt {n-1}$ for all $n$.\\
Then let
 $$\varphi^{\pm}_n= \vp_n ( 1_{B_{n-1}} \pm 1_{\Omega\setminus B_{n-1} }).$$
 Note
 that since $\P(B_{n-1} )=1/2$ we have
$ \varphi^{+}_n \perp \varphi^{-}_n$ for any $n$
and hence
$ \varphi^{+}_n \perp \varphi^{-}_k$ for any $n,k$.
 Then  each of the sequences $\{\varphi^{\pm}_k\mid k\le n\}$
 is a martingale difference sequence with values
 in $\{\pm 1\}$. It is a well known fact 
 (proved by induction as a simple exercise) that
 this forces each to be distributed uniformly over all choices of signs.
 Now let 
 $\{\psi_k\mid k\le 2n\}$ denote the union of the two systems
 $\{\varphi^{+}_k\mid k\le n\}$ and $\{\varphi^{-}_k \mid k\le n\}$.
 Clearly the Sidon constant of  $\{\psi_k\mid k\le 2n\}$
 dominates that of $\{(\varphi^{+}_k+\varphi^{-}_k)/2\mid k\le n\}$.
 But the latter is the system $\{\vp_k 1_{B_{k-1}} \mid k\le n\}$
 as in the preceding proof 
 but with $B_k$ replacing $A_k$.
Since
 $\|S'_n  \|_\infty  \le 1+\sqrt {n-1}$,  
  \eqref{e1}  still holds for this system, so the corollary follows.
\end{proof} 

 \begin{rem} We may clearly replace
 $(\vp_n)$
by an i.i.d. sequence of
complex valued  variables $(z_n)$
uniformly distributed over the unit circle 
of $\C$. For those it is still true
that for any unimodular  sequence $(w_n)$
that is adapted
 (i.e.  $w_n$
is $\cl A_n$-measurable for each $n$)
the sequence $(z_n w_{n-1})$ is
independent and uniformly distributed over the unit circle.
Then the corresponding two sequences
$(\varphi^{\pm}_n)$  are Sidon with constant 1,
and their union is not Sidon
for the same reason as in the preceding corollary.
 \end{rem}
 
 \medskip
 
\n{\bf Problem :}  In \cite{BoLe} Bourgain and Lewko 
show that any $n$-tuple  forming a  $\beta$-subgaussian orthonormal
 system  uniformly bounded by a constant $C$ 
 contains a  subset of cardinality  $\ge \theta n$ with $\theta=\theta(\beta,C)>0$
 that is Sidon  
 with Sidon constant at most $f(\beta,C)$.
They ask whether
 any  such system is actually
  the union of $k(\beta,C)$ Sidon sequences
 with Sidon constant at most $f(\beta,C)$.\\
 Is this true for uniformly bounded martingale difference sequences
 normalized in $L_2$ ?
 
 \medskip
 Although for the example appearing in the proof of Theorem \ref{t1}
 the answer is affirmative (consider e.g. a partition into odd and even $k$'s),
 we believe that a more involved one with values in $\{-1,0,1\}$
 as in \eqref{e7} but with a more subtle
 choice of the predictable sets $A_{n-1}$,  should yield a counterexample.
 
 \medskip
 
 Let $M_n$ be the space of $n\times n$-matrices with complex
 entries, equipped with the usual operator norm on the 
 $n$-dimensional Hilbert space.
 In \cite{Pi3} we consider a non-commutative analogue involving
 a $n\times n$-matrix valued function $\varphi(t)=[\varphi(t)_{ij}]$
 on a probability space $(T,m)$ 
 for which the uniform boundedness
 condition is replaced by
 $$ \|\varphi(t)\|_{M_n} \le C$$
 and we assume
 that $\{\sqrt n \varphi(t)_{ij}\mid 1\le i,j\le n\}$ is $\beta$-subgaussian
 and orthonormal. The prototypical example
 is when $\varphi$ is uniformly distributed over the unitary group.\\
 In this situation we prove in  \cite[Prop. 5.4]{Pi3}
 that  there is a constant $\alpha=\alpha(C,\beta)$
 such that 
 $$\forall a\in M_n \quad
 \tr|a|\le \alpha \sup_{t_1,t_2\in T} |\tr(a\varphi(t_1)\varphi(t_2))| .$$
 In analogy with Theorem \ref{t1} it is natural to wonder
 what is the best  constant $C'_n $ such that
 in the same situation
 $$\forall a\in M_n \quad
 \tr|a|\le C'_n \sup_{t \in T} |\tr(a\varphi(t)  )| .$$
 Clearly the orthonormality  assumption  yields 
 $$\forall a\in M_n \quad
 n^{-1} \tr|a|\le (n^{-1} \tr|a|^2)^{1/2}=   \| \tr(a\varphi(t )  \|_2\le
  \| \tr(a\varphi(t )  \|_\infty=  \sup_{t \in T} |\tr(a\varphi(t ) | .$$
and hence  $C'_n \le n$. \\
It is easy to see that this is asymptotically optimal.
Indeed, consider the following example.
Let $x\mapsto D(x)$ be the mapping taking
an $n\times n$ matrix to its diagonal part.
Let $u$ denote a random $n\times n$ unitary matrix
uniformly distributed over the unitary group.
 Let $(\varphi_1 , \cdots,\varphi_n )$
 be the orthonormal $n$-tuple constructed
 in the proof of Theorem \ref{t1}, of which we 
 keep the notation, namely $\varphi_k=f_k \|f_k\|_2^{-1}$.
 Assuming
$(T,m)$ large enough, we define $\varphi: T \to M_n$
so that $\varphi -D(\varphi)$ and $D(\varphi)$
are independent random variables; we 
make sure that $\varphi -D(\varphi)$
and $u -D(u)$ have the same distribution
and we adjust the diagonal entries
of $D(\varphi)$ so that they have the same distribution
as $(\varphi_1/\sqrt n, \cdots,\varphi_n/\sqrt n)$.
Then for a suitable $\beta$ (independent of $n$)
$\{\sqrt n \varphi(t)_{ij}\mid 1\le i,j\le n\}$ is $\beta$-subgaussian
 and orthonormal.
 However, 
 if $a$
 is the diagonal matrix
 with entries $(\|f_1\|_2, \cdots,\|f_n\|_2)$ we have  on one hand by \eqref{e6}
 $\|\tr(a\varphi)\|_\infty= \|(f_1 + \cdots+f_n)/\sqrt n\|_\infty
 \le c_\vp$, and on the other hand
$\tr|a|\ge n(1+\vp)^{-1}$. Therefore
$$C'_n\ge n(1+\vp)^{-1} c_\vp^{-1}.$$

 \begin{dfn}
 Let $I$ be an index set. Let  $L_1(m'),L_1(m'')$ be arbitrary $L_1$-spaces.
We   say that a  family $(f_n)_{n\in I}$ in 
  $L_1(m'')$ 
 is $c$-dominated  by another one  $(\psi_n)_{n\in I}$ in 
$L_1(m')$ if there is a linear map $u:\ L_1(m')\to L_1(m'')$
with $\|u\|\le c$ such that $u(\psi_n)= f_n$ for all $n\in I$.
\end{dfn}
The following criterion due to 
M. L\'evy  (see \cite{Lev} and \cite[Prop.1.5]{Pi3}) is very useful:  
  a linear map $v: E \to L_1(m'')$
on a subspace $E\subset L_1(m')$ admits an extension
$  u: L_1(m') \to L_1(m'')$ with $\|    u\|\le 1$
  iff
for any finite sequence $(\eta_n)$ in $E$ we have
\begin{equation}\label{L2}\| \sup |v(\eta_n)|\|_{L_1(m'')} \le \| \sup |\eta_n|\|_{L_1(m')}.\end{equation}
If we apply this to $E={\rm span}[\psi_n]$ with $v$
defined by $v(\psi_n)= f_n$, 
this gives us the following criterion: a sequence
$(f_n)_{n\in I}$ in 
  ${L_1(m'')}$ 
 is $c$-dominated    by a sequence  $(\psi_n)_{n\in I}$ in 
  ${L_1(m')}$ iff for any Banach space
  $B$ and any finite sequence $(x_n)$ in $B$ we have
  \begin{equation}\label{L1} \|\sum f_n x_n \|_{L_1(B)} 
  \le c \|\sum \psi_n x_n \|_{L_1(B)} .\end{equation}
  Indeed, it is easy to see that we may restrict consideration
  to the single space $B=\ell_\infty$, in which case
  \eqref{L2} and  \eqref{L1} are identical.
\begin{rem}
The key fact used in \cite{Pi3} is that,
for some numerical constant $K$, any $\beta$-subgaussian sequence
$(\varphi_n)_{n\in \N}$ in $X=L_1(T,m)$
is $K\beta$-dominated  by 
a standard i.i.d. sequence of Gaussian normal variables (on a probability space
$(\Omega',\P')$), denoted by
$(g_n)_{n\in \N}$.
This is essentially due to Talagrand;
see  \cite{Pi3}  for detailed references and comments. It would be interesting to
have a direct simple proof of this fact.

If we assume moreover that the $\beta$-subgaussian sequence
$(\varphi_n)_{n\in \N}$ is uniformly bounded, i.e. that
$\|  \varphi_n\|_\infty \le \alpha$ for all $n$, then,
for some numerical constant $K'$, 
the sequence
$(\varphi_n)_{n\in \N}$  
is $K'(\beta+ \alpha)$-dominated 
by $(\vp_n)$.
This follows from  the solution by Bednorz and Lata\l a
 \cite{BLa}
of Talagrand's Bernoulli conjecture.
 \end{rem}
 
We would like to observe that if $(f_n)$ is a martingale difference sequence
then a very simple proof is available (with an optimal constant).
We start with a special case
of the form $f_n=\vp_n \varphi_{n-1}$
with $\varphi_{n-1}$ depending only on $\vp_1,\cdots,\vp_{n-1}$
satisfying $\|\varphi_{n-1}\|_\infty\le 1$ (which is subgaussian by \eqref{e3}).
This is particularly easy.
 Indeed, for any $y\in [-1,1]$
let $$F(t,y)=(-1) 1_{[0,(1-y)/2]}(t) + (1)1_{((1-y)/2,1]}(t) ,$$ so that
$\int_0^1 F(t,y) dt=  y$ and $   F(t,y)  =\pm 1$.
Let us consider the sequence of random variables $F_n$
defined on $[0,1]^\N\times \{-1,1\}^\N$ by setting
$$F_n( (t_j), (\vp_j)) = \vp_n F(t_{n-1}, \varphi_{n-1}).$$
Let  $u$ be the conditional expectation 
onto the algebra of functions depending on 
the second variable on $[0,1]^\N\times \{-1,1\}^\N$.
Then $u(F_n)=f_n$. Moreover since $(F_n)$ is a martingale
with values in $\pm 1$ it has the same distribution as
$(\vp_n)$ itself. 
In other words, there is
an isometry  $v  :L_1(\Omega,\P) \to L_1([0,1]^\N\times \{-1,1\}^\N)$
such that $v(\vp_n)= F_n$ for all $n$. Considering the composition $uv$,
this shows that $(f_n)$ is 1-dominated    by
$(\vp_n)$, and the latter is easily shown to be 
$c$-dominated  by $(g_n)$ (the latter being, say, in $L_1(\Omega',\P')$) for some numerical constant $c$.

More generally, let
$(\Omega', \cl A',\P')$ be an arbitrary probability space. We have
\begin{lem}\label{mm2} Let $\varphi\in L_1(\Omega', \cl A',\P')$ be with values in $[-1,1]$
and such that $\E \varphi=0$.
 Then for any Banach space
  $B$ and any $x_0,x_1\in B$
   \begin{equation}\label{e30}\E'\|x_0+\varphi x_1\|\le \E\|x_0+\vp_1 x_1\|.\end{equation}
   More generally,
   if $\cl B\subset \cl A'$ is any $\sigma$-subalgebra
   such that $\E^{\cl B}\varphi=0$ we have
   for any $x_0\in L_1(\Omega', \cl B,\P'; B)$
    \begin{equation}\label{e31}\E'\|x_0+\varphi x_1\|\le \E'\E\|x_0+\vp_1 x_1\|.
    \end{equation}
\end{lem}
\begin{proof}  We have
$$x_0+\varphi x_1= \int x_0 + F(t, \varphi) x_1 dt.$$
 and hence by Jensen
 $$\|x_0+\varphi x_1\|\le \int \|x_0 + F(t, \varphi) x_1 \| dt=\|x_0 - x_1 \|(1-\varphi)/2
 +\|x_0 + x_1 \| (1+\varphi)/2.$$
 After integration, we obtain \eqref{e30}.
 To prove \eqref{e31} it suffices to show that
\begin{equation}\label{e32} \E^{\cl B}\|x_0+\varphi x_1\|\le \E^{\cl B} (\|x_0+  x_1\|+\|x_0-  x_1\| )/2,  \end{equation}
or equivalently that for any $A\in \cl B$ with $\P'(A)>0$ we have
\begin{equation}\label{e32'} \P'(A)^{-1}\int_A\|x_0+\varphi x_1\|d\P'\le \P'(A)^{-1}\int_A (\|x_0+  x_1\|+\|x_0-  x_1\| )/2  d\P',  \end{equation}
 Assume that
 $A\in \cl B$ is an atom of $\cl B$. Then $x_0$ is constant on $A$
 and $\E^{\cl B}$ when restricted to $A$ coincides with the average over $A$.
 Thus \eqref{e32'} reduces to \eqref{e31} with $\P'$ replaced
 by $\P'(A)^{-1} \P'_{|A}$.
The case of a general $A\in \cl B$ can be proved by a routine approximation argument
left to the reader.
 \end{proof}
We now show that any real valued martingale difference sequence
with values in $[-1,1]$ is 1-dominated by $(\vp_n)$.

\begin{lem}\label{mm1}
Let $(d_n)$ be a sequence of real valued martingale differences on  
$(\Omega', \cl A',\P')$, i.e.  there are $\sigma$-subalgebras
$\cl A_n\subset \cl A$ ($n\ge 0$) forming an increasing filtration
such that $d_n$ is ${\cl A}_n$-measurable for all
$n\ge 0$ and $\E^{{\cl A}_{n-1} }d_n=0$ for all $n\ge 1 $.
We assume that $\cl A_0$ is trivial (so that $d_0$ is constant).
If $ |d_n|\le 1$ a.s. for any $n$, then there is
an operator $u: L_1(\Omega, \cl A,\P) \to L_1(\Omega', \cl A',\P')$
with $\|u\|= 1$ such that  $u(1)=1$ and $u(\vp_n)=d_n$ for all $n\ge 1 $.
\end{lem}
\begin{proof}  
By the above criterion \eqref{L1}
it suffices to show that for any
Banach space $B$ and   any  finite sequence $(x_n)$ in $B$
we have for any $k$
\begin{equation}\label{e33}\| d_0 x_0+ \sum\nl_1^{k} d_n x_n\|_{L_1(B)}
\le   \| d_0 x_0+  \sum\nl_1^{k} \vp_n x_n\|_{L_1(B)} .\end{equation}
By  \eqref{e31} with $\cl B= \cl A_{k-1}$ and $\varphi=d_k$  we have 
$$\| d_0 x_0+ \sum\nl_1^{k} d_n x_n\|_{L_1(B)}
\le   \| d_0 x_0+  \sum\nl_1^{k-1} d_n x_n + \vp_k x_k\|_{L_1(\P'\times \P;B)} .$$
Now working on the product space
$(\Omega, \cl A,\P) \times (\Omega', \cl A',\P')$
with $\cl B$ equal to $\sigma( \cl A_{k-2} \cup \vp_k) $
we find
$$ \| d_0 x_0+  \sum\nl_1^{k-1} d_n x_n + \vp_k x_k\|_{L_1(\P'\times \P;B)}
\le  \| d_0 x_0+  \sum\nl_1^{k-2} d_n x_n + \vp_{k-1} x_{k-1} + \vp_k x_k\|_{L_1(\P'\times \P;B)}.$$
Continuing in this way we obtain \eqref{e33}.
\end{proof}

\begin{rem}[On the complex valued case in Lemma \ref{mm1}]\label{mm3}
Let $\T=\R/2\pi\Z$ be the (one dimensional) torus.
Consider the sequence $(z_n)_{n\in \N}$ formed 
of the coordinate functions on $\T^{\N} $ equipped with its normalized Haar measure
$\mu$. A priori the complex analogue of the preceding proof,
with $(z_n)$ replacing $(\vp_n)$,
requires to assume that  the martingale under consideration is
a Hardy martingale in the sense described e.g. in \cite[p. 133]{Pi4}.
Indeed, the Poisson kernel is the natural analogue
of the barycentric argument we use for Lemma \ref{mm2}.
Using this, Lemma \ref{mm1} remains valid, with $(z_n)$ replacing $(\vp_n)$, for a
martingale difference sequence $(d_n)$
adapted to the usual filtration on $\T^\N$ such that
for any $n$ the variable $z\mapsto d_n(z_0,\cdots,z_{n-1}, z)$
is either analytic or anti-analytic. \\
Note that 
without any additional assumption
the complex valued case of Lemma \ref{mm1}
fails, simply because the system $(1,\vp_1)$ is not 1-dominated
by $(1,z_1)$. Indeed,  by \eqref{L2} this would imply the inequality
$2=\int \max\{|1+\vp_1| ,|1-\vp_1| \}d\P\le \int \max\{|1+z_1| ,|1-z_1| \}d\mu$,
which clearly fails.
 \end{rem}
 
   The next two remarks will be used at the very end of this paper.
   \begin{rem}
\label{n30}
Let  $(z_n)_{n\in \N}$ and 
$\mu$ on $\T^{\N} $ be as in Remark \ref{mm3}. Consider two sequences $(f^1_n)$
and $(f^2_n)$  in 
an $L_1$-space $X$. We form their ``disjoint union"
$(f_n)$ by setting $f_{2k}= f^2_k$ and
$f_{2k+1}= f^1_k$.
We claim that if $(f^1_n)$ (resp. $(f^2_n)$)  is 
$c_1$-dominated (resp. $c_2$-dominated)   by  $(z_n)$,
then  $(f_n)$ is $(c_1+c_2)$-dominated    by  $(z_n)$.
Actually, the same claim is valid for the disjoint union
of arbitrary families indexed by   sets $I_1$ and $I_2$ (using $(z_n)_{n\in {I_1 \dot\cup I_2}}$
on $\T^{I_1 \dot\cup I_2}$ instead), but the idea is easier to describe with $I=\N$.
Indeed, since $(z_n)$, $(z_{2n})$ and $(z_{2n+1})$ all have the same distribution,
there is  
 $u_j:\ L_1(\T^\N,\mu)\to X$ ($j=1,2$) with $\|u_j\|\le c_j$ such that $u_2(z_{2n})=f^2_n$ 
 and $u_1(z_{2n+1})=f^1_n$.  Let $\E_1$ and  $\E_2$ be the conditional expectations
 on $L_1(\T^\N,\mu)$ with respect to the $\sigma$-algebras generated
respectively by  $(z_{2n+1})$ and $(z_{2n})$.
Then let $u=u_1\E_1+u_2\E_2$.
We have $u(z_n)=f_n$ for all $n$ and $\|u\|\le \|u_1\E_1\|+\|u_2\E_2\|\le c_1+c_2$.
This proves our claim.
   \end{rem} 
 \begin{rem}\label{n10} Let  $(z_n)$ be as in Remark \ref{n30}
 on $(\T^{\N},\mu  )$. Let $(\varphi_n)$ be in $L_\infty(T,m)$.
 We claim that if  $\|\varphi_n\|_\infty\le1$ for all $n$,
 then $(\varphi_n \otimes z_n)$ 
is dominated   by $(z_n)$. 
 Assume first $|\varphi_n|=1$ a.e. for all $n$.
 Then the  translation invariance
 of the distribution of $(z_n)$ shows that
 $(\varphi_n \otimes z_n)$ has the same distribution 
as $(z_n)$, so the claim is obvious in this case.  
Note that any number $\varphi \in \C$ with $|\varphi|\le 1$
is an average of two points on the unit circle. Using this
it is easy to verify the claim. It can also be checked easily using
the criterion in
\eqref{L1}.
 \end{rem}

 We end this paper by an outline of
 a  proof that the union of two Sidon sequences
 is $\otimes^4$-Sidon, more direct than the one
 in \cite{Pi3}. The route we use avoids
 the consideration of randomly Sidon sequences, it is essentially
the commutative  analogue  of the proof in \cite{Pi5},
with the free Abelian group replacing the   free group.
 The key fact for the latter route is still
 the following: 
 \begin{lem}\label{ke1}
Let $(z_n)$  be  in $L_\infty(\T^{\N} ,\mu)$ as in Remark \ref{n30}.
Let $(T,m)$ be a probability space. 
Let $(f_n)$ be  a sequence in $L_1(T,m)$ that is
dominated   by  $(z_n)$. 
 Then any 
sequence  $(\psi_n)$ in $L_\infty(T,m)$ that is both 
uniformly bounded   and  biorthogonal to 
  $(f_n)$
is $\otimes^2$-Sidon. Here biorthogonal means
$$\forall n,m \quad \int \psi_n f_m= \d_{nm}.$$
\end{lem}
 \begin{proof} 
 Let $u: L_1(\T^{\N} ,\mu) \to L_1(T,m)$
 such that $u(z_n)=f_n$.
 Elementary considerations show that
 it suffices to show that the sequence $(u^*(\psi_n))$
 is $\otimes^2$-Sidon. 
 By another elementary argument $(u^*(\psi_n))$
 is biorthogonal to $(z_n)$. 
 Therefore, it suffices to prove this Lemma
for the case $(T,m) =(\T^{\N} ,\mu)$ and $(\psi_n)=(z_n)$.
This is proved in \cite{Pi3} with $(z_n)$ replaced by  an i.i.d. gaussian sequence, using the Ornstein-Uhlenbeck (or Mehler) semigroup. 
Here we may use Riesz products instead.

We claim that
for any $N$ and any $z^0\in  \T^{\N}$ the function 
$F=\sum\nl_1^N z_n^0 z_n \otimes z_n $
admits for any $0<\vp\le 1$  a decomposition 
 $F=t_\vp+r_\vp$ in the algebraic tensor product $L_1(\T^{\N})\otimes L_1(\T^{\N})$ with 
 $$ \|t_\vp\|_{\wedge}=\int |t_\vp(x,y)| d\mu(x)d\mu(y) \le w(\vp)\text{
 and  } 
 \|r_\vp\|_{\vee} \le \vp,$$
 where we have set
 $$ \|r_\vp\|_{\vee} =\sup\nl_{a,b\in B_{L_\infty}} \left|  \int r_\vp(x,y) a(x)b(y) d\mu(x)d\mu(y)\right|  ,$$ and
 where
 $w(\vp)$ is a function depending only on $0<\vp\le 1$
 (and not on $N$ or $z^0$).
 To verify this we fix $z^0$ and  consider in $L_1( \T^{\N}\times \T^{\N}, \mu\times \mu   )$ the Riesz product
 $$\nu_\vp(z,z')= \prod\nl_1^N (1+\vp \Re(z_n^0 z_nz_n'))=\sum\nl_{\alpha \subset [1...N] } (\vp/2)^{|\alpha|}  \prod_{n\in \alpha} (z_n^0 z_nz_n'+\ovl{z_n^0 z_nz_n'}).$$
We  will view the tensors in  $L_1(\T^{\N})\otimes L_1(\T^{\N})$
as functions of   $(z,z')\in \T^{\N}\times \T^{\N}$.
 Note
    \begin{equation}\label{e41}\nu_\vp(z,z') =\sum\nl_{\alpha \subset [1...N] } (\vp/2)^{|\alpha|}  \sum_{\beta\subset \alpha}\  \prod_{n\in \beta}  z_n^0 z_nz_n' \prod_{n\in [1...N]\setminus \beta} \ovl{z_n^0 z_nz_n'}.\end{equation}
 Observe that the terms of the latter sum are orthogonal.
 Without trying to optimize (see \cite{Pi3} for a discussion of the optimal logarithmic
 growth for $w$) we set
 $$t'_\vp =   ( \nu_\vp-\nu_0)/\vp.$$
 Note that (since $\nu_\vp\ge 0$ and hence $\|\nu_\vp\|_1=1$) we have $\|t'_\vp\|_{\wedge}\le 2/\vp$.
 Let $r'_\vp= \sum\nl_1^N \Re(z_n^0 z_nz_n')) -t'_\vp$.
  By the orthogonality in the sum \eqref{e41}  one checks that $\|r'_\vp\|_{\vee}\le \vp/2$.
 This gives us the desired decomposition
  but,  instead of $\sum\nl_1^N {z_n^0 z_nz_n'}$, we are decomposing the sum
 $$ \sum\nl_1^N \Re(z_n^0 z_nz_n'))=(1/2)\sum\nl_1^N {z_n^0 z_nz_n'}  + (1/2)\sum\nl_1^N \ovl{z_n^0 z_nz_n'} .$$
 To remove the second term  we introduce an extra variable
 $\omega\in \T$ that acts on $\T^\N$ by multiplication ( i.e. $\omega (z_n)= (\omega z_n)  $) and we define (here $m_\T$ is normalized Haar measure on $\T$)
 $$t_\vp(z,z')=2\int  \bar{\omega} t'_\vp(\omega z,z') dm_\T(\omega)
 \text{   and   } r_\vp(z,z')=2\int  \bar{\omega} r'_\vp(\omega z,z') dm_\T(\omega).$$
    This gives us $\| t_\vp \|_{\wedge}\le 4/\vp$ and $\| r_\vp \|_{\vee}\le \vp$. Moreover we have
    $$(1/2)\sum\nl_1^N {z_n^0 z_nz_n'}= (1/2) t_\vp+ (1/2) r_\vp$$
     which proves the claim  with $w(\vp)\le 4/\vp$.
 
 We can now complete the proof.
 Let $(a_n)$ be a scalar sequence. 
 Let $\Psi= \sum\nl_1^N a_n \psi_n \otimes \psi_n$.
 Choosing $z_n^0$ so that $z_n^0 a_n=  |a_n|$ we 
  have $$\langle \Psi,F \rangle =\sum z_n^0 a_n=\sum |a_n|,$$
  and hence
  $\sum |a_n|=\langle \Psi, t_\vp\rangle + \langle \Psi, r_\vp\rangle $
  which leads to 
 $$\sum |a_n| \le   \|\Psi\|_\infty w(\vp) +\sum |a_n| \vp (\sup\nl_{1\le n\le N} \|\psi_n\|_\infty^2).$$
 To conclude, we set $C'=\sup\nl_{n\ge 1} \|\psi_n\|_\infty$ and we choose, say,
 $\vp= 1/2C'^2$. We have then
 $$\sum |a_n| \le 2 w(\vp)  \|\Psi\|_\infty.$$
 \end{proof}

 Let us say that a  bounded set $S$ in $  L_\infty(T,m)$ is Sidon 
  with constant $C$ if for any finitely supported  function $x: S \to \C$ 
  we have
  $\sum\nl_{\varphi\in S} |x(\varphi)| \le C \|\sum x(\varphi) \varphi\| .$
 If $(\varphi_n)$ is an enumeration of $S$, this is the same
 as $\sum\nl_{n\in \N} |x(n)| \le C \|\sum\nl_{n\in \N} x(n) \varphi_n\| .$
 Similarly we  extend the term 
 $\otimes^4$-Sidon to sets in $  L_\infty(T,m)$.

 For the convenience of the reader we give a slightly more direct proof of
 the following result from \cite{Pi3}, which generalizes Drury's theorem.
  \begin{thm}\label{n1}
Let $\Lambda_1=\{\varphi^1_n\mid n\in I(2)\}$ and $
\Lambda_2=\{\varphi^2_n\mid n\in I(1)\}$ be two Sidon sets (indexed by sets
$I(1),I(2)$)
in $  L_\infty(T,m)$, with constants $C_1,C_2$.
Assume that  $ \Lambda_1  \perp   \Lambda_2$ in $L_2(m)$ and  
there are $C'_1,C'_2,\d>0 $
such that
$$\forall n \quad \d\le \|\varphi^1_n\|_2\le \|\varphi^1_n\|_\infty\le C'_1  \text{  and  } \d\le \|\varphi^2_n\|_2\le\|\varphi^2_n\|_\infty\le C'_2 .$$
Then the union $ \Lambda_1  \cup
    \Lambda_2$
    is $\otimes^4$-Sidon with a constant $C$ depending only
    on  $C_1,C_2,C'_1,C'_2,\d$.
     \end{thm}
\begin{proof} We assume for simplicity that the sets
are sequences indexed by $\N$.
By homogeneity (changing $C'_1,C'_2$ accordingly)
we may assume that $\|\varphi^1_n\|_2=\|\varphi^2_n\|_2=1$
for all $n$. 
 Let $E_j\subset L_\infty(T,m)$ be the norm closed span
 of $(\varphi^j_n)$ ($j=1,2$).
 Consider the linear mapping 
 $T_j: E_j \to L_\infty(\T^{\N})$ such that  $T_j(\varphi^j_n )= z_n$.
 By assumption $\|T_j\|\le C_j$. By the injectivity of $L_\infty$-spaces
$T_j$ has   an extension $\tilde T_j: L_\infty(T,m) \to L_\infty(\T^{\N})$ 
 such that $\tilde {T_j}_{| E_j} =T_j$ and  $\| \tilde T_j  \|=\|T_j\|\le C_j$.
 We introduce the operator
 $ {\cl T} : L_\infty(T,m) \to L_\infty(\T^{\N}\times \T^{\N})$ defined by
 $${\cl T} (  f)(z,z')= \tilde {T_1}(f) (z)   + \tilde {T_2}(f) (z').  $$
 Then $\| {\cl T} \|\le C_1+C_2$. 
 The operator $ {\cl T}  \otimes id_{L_\infty(T,m)} $ clearly extends
 to an  bounded operator
 $$W: L_\infty(T\times T) \to L_\infty(\T^{\N}\times \T^{\N}\times T),$$
 satisfying  $\|W\|\le \| {\cl T} \|\le C_1+C_2$.
 
 We claim that the collection
 $$ \cl U= \{W(\varphi_n^1\otimes \varphi_n^1)\} \cup   \{W(\varphi_n^2\otimes \varphi_n^2)\} $$
 is biorthogonal to
 $$\cl V= \{\ovl{z_n\otimes 1\otimes \varphi_n^1}\} \cup   \{\ovl{1\otimes z_n\otimes \varphi_n^2}\}. $$
 Indeed, 
 note
 $W(\varphi_n^1\otimes \varphi_n^1) \subset L_\infty(\T^{\N}\times \T^{\N}) \otimes \varphi_n^1$ and $W(\varphi_n^2\otimes \varphi_n^2) \subset L_\infty(\T^{\N}\times \T^{\N}) \otimes \varphi_n^2$. Therefore, by our $L_2(m)$-orthogonality assumption
 $$\forall n,m\quad W(\varphi_n^1\otimes \varphi_n^1) \perp {1\otimes z_m\otimes \varphi_m^2}\text{  and  } W(\varphi_n^2\otimes \varphi_n^2) \perp { z_m\otimes 1\otimes \varphi_m^2}.$$
  Moreover,  if we set $\xi_n^1= \tilde {T_2}(\varphi_n^1) $ we have
  $$W(\varphi_n^1\otimes \varphi_n^1) = {\cl T}(\varphi_n^1)\otimes \varphi_n^1
  = z_n \otimes 1 \otimes \varphi_n^1 +  1\otimes \xi_n^1 \otimes \varphi_n^1,$$
  which shows that $(W(\varphi_n^1\otimes \varphi_n^1))$
  is biorthogonal to $ \{\ovl{z_n\otimes 1\otimes \varphi_n^1}\}$.
 Similarly $(W(\varphi_n^2\otimes \varphi_n^2))$
  is biorthogonal to $ \{\ovl{ 1\otimes z_n\otimes\varphi_n^2}\}$.
  This proves the claim.\\
 By Remarks \ref{n10} and \ref{n30}, the family  
  $ \cl V
   =\{\ovl{z_n\otimes 1\otimes \varphi_n^1}\} \cup   \{\ovl{1\otimes z_n\otimes \varphi_n^2}\} 
  $ is dominated in $L_1(\T^{\N}\times \T^{\N}\times T)$ by the sequence $(z_n)$.
    By  Lemma \ref{ke1}
we conclude that $\cl U$ is $\otimes^2$-Sidon in $L_\infty(\T^{\N}\times \T^{\N}\times T) $.
Since $W$ is bounded this implies
that $\{\varphi_n^1\otimes \varphi_n^1\} \cup   \{\varphi_n^2\otimes \varphi_n^2\} $
is also $\otimes^2$-Sidon in $L_\infty(T\times T) $.
Consequently  $ \Lambda_1  \cup
    \Lambda_2$
    is $\otimes^4$-Sidon in $L_\infty(T,m)$.
    The assertion about the constant $C$ is easy to check by going over the various steps.
\end{proof}

 \n\textit{Acknowledgement.} Thanks  to  Bernard Maurey for useful communications,
 and to the referee for his/her careful reading.

 \end{document}